\documentclass[12pt,reqno]{amsart}
\usepackage{amsmath, amsthm, amssymb, stmaryrd}
\usepackage{cleveref}
\usepackage{bbm}

\usepackage{geometry} 

\usepackage{tikz,ifthen}
\usepackage{graphicx}
\usetikzlibrary{patterns}

\crefformat{section}{\S#2#1#3} 
\crefformat{subsection}{\S#2#1#3}
\crefformat{subsubsection}{\S#2#1#3}

\newtheorem{theorem}{Theorem}[section]
\newtheorem{lemma}[theorem]{Lemma}
\newtheorem{corollary}[theorem]{Corollary}

\theoremstyle{definition}
\newtheorem{definition}[theorem]{Definition}

\theoremstyle{remark}

\numberwithin{equation}{section}

\newcommand{\nmod}[1]{\,\,\text{\rm mod}\,\,#1}

\newcommand{\quand}{\quad \text{and} \quad}
\def\rank{{ \normalfont \text{rank} \,}}
\def\im{{ \normalfont \text{im}\,}}
\def\rem{{\normalfont \text{rem}}}
\newcommand{\floor}[1]{\left \lfloor #1 \right \rfloor}
\newcommand{\ceil}[1]{\left \lceil #1 \right \rceil}
\def\bfa{{\mathbf a}}
\def\bfb{{\mathbf b}}
\def\bfc{{\mathbf c}}
\def\bfd{{\mathbf d}} \def\bfe{{\mathbf e}}

\def\bfv{{\mathbf v}}

\def\bfx{{\mathbf x}}
\def\bfy{{\mathbf y}}

\def\bfJ{{\mathbf J}}

\def\bfZ{{\mathbf Z}}

\def\N{{\mathbb N}}  
\def\R{{\mathbb R}}
\def\Z{{\mathbb Z}}

\def\grm{{\mathfrak m}}\def\grM{{\mathfrak M}}
\def\grN{{\mathfrak N}}

\def\grS{{\mathfrak S}}

\def\a{{\alpha}} \def\bfalp{{\boldsymbol \alpha}} 
\def\b{{\beta}}  
\def\gam{{\gamma}} 
\def\bfgam{{\boldsymbol \gam}}

\def\bfxi{{\boldsymbol \xi}}

\def\eps{\varepsilon}

\def\le{\leqslant} \def\ge{\geqslant}

\def\d{{\,{\rm d}}}

\begin{document}
\title[K-Multimagic Squares ]{A Circle Method Approach to \\ K-Multimagic Squares}
\author[Daniel Flores]{Daniel Flores}
\address{Department of Mathematics, Purdue University, 150 N. University Street, West 
Lafayette, IN 47907-2067, USA}
\email{flore205@purdue.edu}
\date{}
\dedicatory{}
\begin{abstract}
In this paper we investigate $K$\emph{-multimagic squares} of order $N$, these are $N \times N$ magic squares which remain magic after raising each element to the $k$th power for all $2 \le k \le K$. Given $K \ge 2$, we consider the problem of establishing the smallest integer $N_2(K)$ for which there exists \emph{nontrivial} $K$-multimagic squares of order $N_2(K)$. Previous results on multimagic squares show that $N_2(K) \le (4K-2)^K$ for large $K$. Here we utilize the Hardy-Littlewood circle method and establish the bound 
\[N_2(K) \le 2K(K+1)+1.\] 

Via an argument of Granville's we additionally deduce the existence of infinitely many \emph{nontrivial} prime valued $K$-multimagic squares of order $2K(K+1)+1$.
\end{abstract}
\maketitle

\section{Introduction}\label{sec:intro}

A $N \times N$ matrix $\bfZ = (z_{i,j})_{1 \le i,j \le N}$ is a \emph{magic square} of order $N$ if the sum of the entries in each of its rows, columns, and two main diagonals are equal. The concept of magic squares has fascinated mathematicians and laymen for thousands of years. Although the study of these objects dates back millennia, there are still many unresolved problems concerning magic squares. One of the most famous problems in this area concerns the existence of a $3 \times 3$ magic square where each element is a distinct square number. This problem became popularized by Martin Gardner in 1996, and was listed in Richard Guy’s book, \emph{Unsolved Problems in Number Theory}. 

One may also investigate problems related to so called \emph{multimagic} squares. Given $K \ge 2$ we say a matrix $\bfZ \in \Z^{N \times N}$ is a $K$-multimagic square of order $N$ or \textbf{MMS}$(K,N)$ for short if the matrices
\[\bfZ^{\circ k} := (z_{i,j}^k)_{1 \le i,j \le N},\]
remain magic squares for $1 \le k \le K$. Before we can state our problem of interest we must first discuss \emph{trivial} multimagic squares.

It is clear any multiple of the $N \times N$ matrix of all ones is trivially a \textbf{MMS}$(K,N)$ for every $K \ge 2$. However, this is not the only family of ``trivial'' multimagic squares one must consider. Suppose $\bfZ$ is an $N \times N$ matrix in which every row, column, and both main diagonals contain precisely $N$ distinct symbols. Such matrices are known as doubly diagonalized Latin squares of order $N$, or \textbf{DDLS}($N$) for short. \textbf{DDLS}($N$) are known to exist for all $N \ge 4$, see \cite{Gergely1974}. Then for $N \ge 4$ any mapping of these $N$ symbols to the integers yields a \textbf{MMS}$(K,N)$ for every $K \ge 2$. Consideration of these ``trivial'' \textbf{MMS}$(K,N)$ motivates the following definition.

\begin{definition}
    For $K \ge 2$ and $N \in \N$ a \textbf{MMS}$(K,N)$ is called \emph{trivial} if it utilizes $N$ or less distinct integers.
\end{definition}

We begin by observing that for $N = 1$ or $2$ the only \textbf{MMS}$(K,N)$ are those with every element being equal. Additionally it is known that every $3 \times 3$ magic square may be parametrized by three variables $a,b,c$ as follows 
\begin{equation*}
\bfZ(a,b,c) = \left[\begin{matrix}
    c-b & c+(a+b) & c-a \\
    c-(a-b) & c & c+(a-b) \\
    c+a & c-(a+b) & c+b \\
\end{matrix} \right].
\end{equation*}
Solving for $a,b,c$ such that $\bfZ(a,b,c)^{\circ 2}$ is a magic square one sees that the only solutions are those with $a = b = 0$. Thus for $K \ge 2$ and $1 \le N \le 3$ we conclude that the only \textbf{MMS}$(K,N)$ are trivial.

In this paper we investigate the minimal value $N_2(K)$ for which there exists a nontrivial \textbf{MMS}$(K,N_2(K))$. This type of question has been considered in the past by several authors (see \cite{BoyerSite,Derksen2007,TrumpSite,Zhang2013,Zhang2019}) via constructions of \emph{normal} multimagic squares. A multimagic square is said to be \emph{normal} if its elements consists of the integers $1,2,\ldots,N^2$. We give a brief overview of the best known results for nontrivial \textbf{MMS}$(K,N)$ below.
\begin{figure}[htbp]
    \centering
    \begin{tabular}{|c|c|c|} 
\hline  
$K$ & Upper bound on $N_2(K)$ & Attributed to \\ 
\hline  
2 & $6$ & J. Wroblewski \cite{BoyerSite} \\
3 & $12$ & W. Trump \cite{TrumpSite} \\
4 & $243$ & P. Fengchu \cite{BoyerSite} \\
5 & $729$ & L. Wen \cite{BoyerSite} \\
6 & $4096$ & P. Fengchu \cite{BoyerSite} \\
$K \ge 2$ & $(4K-2)^K$ & Zhang, Chen, and Li \cite{Zhang2019}\\
\hline  
\end{tabular}
\end{figure}

In this paper we establish via the Hardy-Littlewood circle method the following result.
\begin{theorem}\label{Thmnontriv}
    $N_2(K) \le 2K(K+1)+1$ for $K \ge 2$.
\end{theorem}
This beats previously known results as soon as $K \ge 4$ and shows that $N_2(K)$ grows at most quadratically in $K$ rather than potentially exponential in $K$. One may prove an analogous statement for prime valued \textbf{MMS}$(K,N)$ by reapplying the entirety of the circle method where we detect prime solutions instead of integer solutions. This, however, is not necessary as via an argument due to Granville in \cite{Granville2008} one may apply the Green-Tao theorem and deduce the following.
\begin{corollary}
    Given $K \ge 2$ there exists infinitely many nontrivial prime valued \textbf{MMS}$(K,N)$ for every $N > 2K(K+1)$.
\end{corollary}

One may potentially generalize these results to multimagic $d$-dimensional hypercubes given one works out the analogues of sections \ref{sec:multmagicsquare} and \ref{sec:nonsingsoln} for the $d$-dimensional case. Assuming one does this, we expect that the circle method will yield the bound
\[N_d(K) \ll_d K^2,\]
where $N_d(K)$ is the minimal number for which there exists multimagic $d$-dimensional hypercubes which use $N^{d-1}+1$ or more distinct integers.

\section{Overview of the Paper}\label{sec:overview}
Given $K \ge 2$ and $N > 2K(K+1)$, we let $M_{K,N}(P)$ denote the number of \textbf{MMS}$(K,N)$ with entries satisfying
\[\max_{1 \le i,j \le N}|z_{i,j}| \le P.\]
One easily sees that the number of trivial \textbf{MMS}$(K,N)$ counted by $M_{K,N}(P)$ is $O_N(P^N)$, thus if one wishes to establish the existence of infinitely many nontrivial \textbf{MMS}$(K,N)$ it is enough to show that
\begin{equation}\label{nontrivlimiteq}
    \frac{M_{K,N}(P)}{P^N} \to \infty \quad \text{as}\quad N \to \infty.
\end{equation}

We begin by considering a general diagonal system of equations in differing degrees. Let $C = (c_{i,j})_{\substack{1 \le i \le r \\ 1 \le j \le s}} \in \Z^{r \times s}$ be given, and consider the diagonal system
\begin{equation}\label{GenSys}
    \sum_{1 \le j \le s}c_{i,j}x_j^k = 0 \quad  (1 \le i \le r, \quad 1 \le k \le K).
\end{equation}
We define $R_K(P;C)$ to be the number of solutions $\bfx \in \Z^s$ to (\ref{GenSys}) where $\max_{j}|x_j| \le P$. This class of problems has been investigated in the past by various authors (see \cite{Brandes2017,Brandes2021}). However, in their application of the circle method they require the $r \times s$ matrix of coefficients $C$ to be \emph{highly nonsingular}, i.e., for all $J \subset \{1,\ldots,s\}$ with $|J| = r$ one should have
\begin{equation*}
    \det \left(c_{i,j}\right)_{\substack{1 \le i \le r \\ j \in J}} \neq 0.
\end{equation*}
Upon examination of these methods, however, one sees that a slightly weaker condition on the matrix $C$ would suffice. This weaker condition has been used previously by Br\"{u}dern and Cook \cite{Brudern1992}, who investigated diagonal systems of a fixed degree $k$. This is in fact crucial because, upon investigation, one sees that the matrix of coefficients associated to multimagic squares is certainly not highly nonsingular. 

These considerations lead us to define a notion of when a matrix $C$ \emph{dominates} a function. We establish in section \ref{sec:circmethod} an asymptotic formula for $R_K(P;C)$ provided $C$ dominates an appropriate function. Before stating our results we must establish some notation. 

For a given $r \times s$ matrix $C = [\bfc_1,\ldots,\bfc_s]$ and any set $J \subset \{1,\ldots,s\}$, we denote by $C_J$ the submatrix of $C$ consisting of the columns $\bfc_j$ where $j \in J$. For any $a \in \Z$ and $b \in \N$ we denote by $ \rem (a,b)$ the remainder of $a$ modulo $b$ considered as an integer between $0$ and $b-1$. 
\begin{definition}
    We say that a matrix $C$ dominates a function $f:\N \to \R^+$ whenever the inequality
\begin{equation*}
    \text{rank}(C_J) \ge \min\left\{  f(|J|),r\right\},
\end{equation*}
holds for all $J \subset \{1,\ldots,s\}$.
\end{definition}

\begin{theorem}\label{CircMethThm}
    Let $K \ge 2$ and suppose that $C \in \Z^{r \times s}$ satisfies $s > rK(K+1)$. Then, if $C$ dominates the function
    \[F(x) = \max\left\{ \frac{x- \rem (s,r)}{\floor{\frac{s}{r}}}, \frac{x- \rem (s-1,r)}{\floor{\frac{s-1}{r}}} \right\},\]
    one has that
    \[R_K(P;C) = P^{s-\frac{rK(K+1)}{2}}\left(\sigma_K(C) + o(1)  \right),\]
    where $\sigma_K(C) \ge 0$ is a real number depending only on $K$ and $C$. Additionally $\sigma_K(C) > 0$ if there exists nonsingular real and $p$-adic solutions to the system (\ref{GenSys}).
\end{theorem}

Theorem \ref{CircMethThm} can be seen as a relaxation on the condition that $C$ be highly nonsingular, which would be equivalent to $C$ dominating the identity function.

In section \ref{sec:multmagicsquare} we establish the existence of a matrix $C^{\text{magic}}_{N} \in \{-1,0,1\}^{2N \times N^2}$ for which $M_{K,N}(P) = R_K(P;C^{\text{magic}}_{N})$ and prove that this matrix dominates $F(x)$ from Theorem \ref{CircMethThm}. This is done via a combinatorial argument and understanding the underlying linear system associated to the matrix $C^{\text{magic}}_{N}$. In section \ref{sec:nonsingsoln} we establish that $\sigma_K(C) > 0$. This is done by showing that a \textbf{DDLS}$(N)$ with distinct integer symbols is a nonsingular integer solution to the system (\ref{GenSys}) with $C = C^{\text{magic}}_{N}$.

Thus from the conclusions made in sections \ref{sec:multmagicsquare} and \ref{sec:nonsingsoln} in combination with Theorem \ref{CircMethThm} we deduce the following.
\begin{theorem}\label{MainThm}
    For $K \ge 2$ and $N > 2K(K+1)$ there exists a constant $c > 0$ for which one has the asymptotic formula 
    \[M_{K,N}(P) \sim cP^{N(N-K(K+1))}.\]
\end{theorem}
Whence by (\ref{nontrivlimiteq}) we finally establish Theorem \ref{Thmnontriv}.

\section{Application of the Circle Method}\label{sec:circmethod} 
Our basic parameter, $P$, is always assumed to be a large positive integer. Whenever $\eps$ appears in a statement, either implicitly or explicitly, we assert that the statement holds for every $\eps>0$. Implicit constants in Vinogradov's notation $\ll$ and $\gg$ may depend on $\varepsilon$, $r$, $s$, and the elements of the matrix $C$. 

We also make use of the vector notation $\bfx = (x_1,\ldots,x_r)$ where $r$ is dependent on the context of the argument. Whenever the notation $|\bfx|$ is used for a vector or matrix we mean the maximal absolute value of the elements in $\bfx$. Any statement regarding matrices is to be understood componentwise. With this in mind, whenever we write an inequality involving a matrix $\bfa$, such as $X \le \bfa \le Y$, we mean that the inequality $X \le a_{i,j} \le Y$ holds for all elements of $\bfa$. Similarly, when $q \in \N$ and $\bfa$ is an integer matrix we write $(q,\bfa)$ to denote the simultaneous greatest common divisor $q$ and the elements of $\bfa$. 

We use $\|x\|$ to refer to the distance to nearest integer of $x$. We will occasionally define functions of matrices which will change depending on the number of columns of the matrix. As is conventional in analytic number theory, we write $e(z)$ for $e^{2 \pi i z}$. Additionally, we write $[n]$ to denote the set of integers from $1$ of to $n$. \par

We now proceed to define our basic exponential generating functions necessary for our application of the Hardy-Littlewood circle method. Whenever we make references to a collection of $rK$ many variables, say $\bfalp$, we will represent these as an $r \times K$ matrix
\[
\bfalp = \left[\begin{matrix}
    \a_{1,1} & \cdots & \a_{1,K} \\
    \vdots & \ddots & \vdots \\
    \a_{r,1} & \cdots & \a_{r,K}
\end{matrix}\right] = \left[ \bfalp_{1},\ldots, \bfalp_{K} \right].
\]
Thus whenever $\d \bfalp$ shows up we mean 
\[\prod_{\substack{1 \le i \le r \\ 1 \le k \le K}} d\a_{i,k}.\]
For any matrix $C = [\bfc_1,\cdots,\bfc_s]$ of dimension $r \times s$ we define
\[f_K(\bfalp;C) = \prod_{1 \le j \le s}\sum_{|x| \le P}e \left(\sum_{\substack{1 \le k \le K}}\left(\bfalp_k \cdot \bfc_j\right) x^k  \right),\]
\[S_K(q,\bfa;C) = \prod_{1 \le j \le s}\sum_{1 \le u \le q} e \left(\frac{1}{q}\sum_{\substack{1 \le k \le K}}\left(\bfa_k \cdot \bfc_j\right) u^k  \right),\]
and
\[ I_K(P,\bfgam;C) = \prod_{1 \le j \le s} \int_{-P}^{P} e\left( \sum_{\substack{1 \le k \le K}}\left(\bfgam_k \cdot \bfc_j\right) z^k \right) \d z.\]

Then by orthogonality it follows that
\[R_{K}(P;C) = \int_{[0,1)^{r \times K}} f_K(\bfalp;C) \d\bfalp.\]
For any $0 < Q \le P$ we define our major arcs to be
\[\grM(Q) = \bigcup_{\substack{0 \le \bfa \le q \le Q \\ (q,\bfa) = 1}}\grM(Q;q,\bfa),\]
where
\[\grM(Q;q,\bfa) = \{\bfalp \in [0,1)^{r \times K}: |q\a_{i,k} - a_{i,k}| \le QP^{-k}\}.\]
Similarly we define the minor arcs to be $\grm(Q) = [0,1]^{r \times K} \backslash \grM(Q)$. 
We also define
\[\grS_K(Q;C) = \sum_{\substack{0 \le \bfa \le q \le Q \\ (q,\bfa) = 1}}q^{-s}S_K(q,\bfa;C),\] 
and 
\[J_K(Q,P;C) = \int_{\substack{\bfgam \in \R^{r \times K} \\ |\bfgam_{k}| \le QP^{-k}}} I_K(P,\bfgam;C) \d\bfgam.\]
\begin{lemma}\label{LemmaMajorMinorDecomp}
    Suppose that $Q = P^{\delta}$ for some $ 0 < \delta <
    \frac{1}{3+2rK}$ and one has the bound
    \begin{equation}\label{minorbndasump}
        \int_{\grm(Q)} |f_K(\bfalp;C)| \d\bfalp = o \left(P^{s-\frac{rK(K+1)}{2}} \right).
    \end{equation}
    Then
    \[R_K(P;C) = \grS_K(Q;C)J_K(Q,P;C) + o \left(P^{s-\frac{rK(K+1)}{2}} \right).\]
\end{lemma}
\begin{proof}
    We begin by defining a slightly larger set of major arcs
    \[\grN(Q) = \bigcup_{\substack{0 \le \bfa \le q \le Q \\ (q,\bfa) = 1}}\grN(Q;q,\bfa),\]
    where
    \[\grN(Q;q,\bfa) = \left\{\bfalp \in [0,1)^{r \times K} : |\a_{i,k} - a_{i,k}/q| \le QP^{-k}\right\} .\]
    Observe that $[0,1)^{r \times K} \backslash \grN(Q) \subset \grm(Q)$, whence from (\ref{minorbndasump}) it follows that
    \[\int_{[0,1)^{r \times K} \backslash \grN(Q)} |f_K(\bfalp;C)| \d\bfalp = o \left(P^{s-\frac{rK(K+1)}{2}} \right).\]
    Thus we have
    \begin{align*}
        R_K(P;C) &= \int_{\grN(Q)}f_K(\bfalp;C)\d\bfalp + o \left(P^{s-\frac{rK(K+1)}{2}} \right).
    \end{align*}

    Via standard methods it follows easily that for $\bfalp \in \grN(Q;q,\bfa)$ one has that
    \[f_K(\bfalp;C) = q^{-s} S_K(q,\bfa;C) I(P,\bfalp-\bfa/q;C) + O(Q^2 P^{s-1}).\]
    Whence upon noting that $\text{mes}(\grN(Q)) \ll Q^{2rK+1}P^{-\frac{rK(K+1)}{2}}$ one deduces
    \begin{align*}
        \int_{\grN(Q)}f_K(\bfalp;C)d\bfalp &= \grS_K(Q;C)J_K(Q,P;C) + O\left( Q^{3+2rK}P^{s-\frac{rK(K+1)}{2}-1} \right) \\
        &= \grS_K(Q;C)J_K(Q,P;C) + o \left(P^{s-\frac{rK(K+1)}{2}} \right).
    \end{align*}
\end{proof}

\subsection{The Minor Arcs}\label{sec:minor}

As is evident from Lemma \ref{LemmaMajorMinorDecomp}, we must establish an adequate bound over the minor arcs. For this we make use of the arguments of \cite{Brandes2017} without assuming our matrix is highly nonsingular. We will instead make use of the assumption that $C$ dominates the function 
\begin{equation}\label{DomFunc}
    F(x) = \max\left\{ \frac{x- \rem (s,r)}{\floor{\frac{s}{r}}}, \frac{x- \rem (s-1,r)}{\floor{\frac{s-1}{r}}} \right\},
\end{equation}
to derive a Weyl type inequality and a mean value bound. 
\begin{definition}
    We say a matrix $C$ of dimensions $r \times rn$ is partitionable if there exists a partition $\bigsqcup_{1 \le l \le n}J_l = [rn]$, for which one has
    \[\rank (C_{J_l}) = r \text{ for all }1 \le l \le n.\]
\end{definition}
In the proofs that follow we make use of the property that our matrix $C$ contains a partitionable submatrix of size $r \times r\floor{s/r}$ and given any $j_0 \in [s]$ we have that $C_{[s]\backslash \{j_0\}}$ contains a partitionable submatrix of size $r \times r\floor{(s-1)/r}$. This may be deduced from the property that $C$ dominates (\ref{DomFunc}) in conjunction with  \cite[Lemma 1]{Low1988}.

\begin{lemma}\label{LemmaWeylTypeBound}
    Let $s \ge r$ and suppose $C \in \Z^{r \times s}$ dominates the function
    \[\frac{x- \rem (s,r)}{\floor{\frac{s}{r}}}.\]
    Then there exists $j_0 \in [s]$ and $\sigma > 0$ for which
    \[\sup_{\bfalp \in \grm(Q)} |f_K(\bfalp;\bfc_{j_0})| \ll PQ^{-\sigma }.\]
\end{lemma}
\begin{proof}
    By \cite[Lemma 1]{Low1988}, $C$ contains a partitionable $r \times rn$ submatrix for all $n \le \floor{\frac{s}{r}}$. Since $s \ge r$ we may take $n = 1$, whence without loss of generality we may assume that the first $r$ columns of $C$ are linearly independent. Let $\sigma < 1/(2K)$ and define
    \[\b_{j,k} = \bfc_j \cdot \bfalp_k.\]
    Now suppose that 
    \[ \left|\sum_{1 \le x \le P}e \left(\sum_{\substack{1 \le k \le K}}\left(\bfalp_k \cdot \bfc_j\right) x^k  \right)\right| > PQ^{-\sigma} \]
    for all $1 \le j \le r$. Then by \cite[Lemma 2.4]{Parsell2002} one has that there exists $q \ll Q^{2K\sigma}$ such that 
    \begin{equation}\label{ParsellResult}
        \left\|q \b_{j,k} \right\| \ll Q^{2K\sigma}P^{-k} \quad \text{for all $1 \le j \le r$ and $1 \le k \le K$}.
    \end{equation}
    Since the first $r$ columns of $C$ are linearly independent one has that there exists vectors $\bfv_i$ for $1 \le i \le r$ which satisfy
    \[\a_{i,k} =  \bfv_i \cdot (\b_{1,k},\ldots,\b_{r,k}),\]
    where $[\bfv_1,\ldots,\bfv_r] = \left([\bfc_1,\ldots,\bfc_r]^T\right)^{-1}$. It is important to note here that the elements of $[\bfv_1,\ldots,\bfv_r]$ are rational. Let $L$ be an integer such that $L[\bfv_1,\ldots,\bfv_r]$ is an integer matrix, we then have that
    \[\|L q\a_{i,k}\| \le \sum_{1 \le j \le r}\|L v_{i,j}q\b_{j,k}\| \leq \sum_{1 \le j \le r}|L v_{i,j}| \|q\b_{j,k}\| \ll (Q/L)^{2K\sigma}P^{-k} ,\]
    where the last inequality comes from (\ref{ParsellResult}) and noting that $L \asymp 1$. We deduce that $L\bfalp \in \grM(Q/L)$ for large enough $Q$ but since $L$ is an integer this is equivalent to $\bfalp \in \grM(Q)$. Thus, by the contrapositive, it must be the case that for some $1 \le j_0 \le r$ we have the desired bound.
\end{proof}
\begin{lemma}\label{LemmaMeanValueBound}
    Let $s \ge rK(K+1)$ and suppose $C \in \Z^{r \times s}$ dominates 
    \[\frac{x- \rem (s,r)}{\floor{\frac{s}{r}}}.\]
    Then one has the bound
    \[\int_{[0,1)^{r \times K}} |f_K(\bfalp;C)|\d\bfalp \ll P^{s - \frac{rK(K+1)}{2} + \eps}.\]
\end{lemma}
\begin{proof}
    By \cite[Lemma 1]{Low1988}, we deduce that $C$ contains a partitionable 
    \[r \times rK(K+1)\] 
    submatrix. Then upon utilizing the trivial bound on exponential sums we may without loss of generality suppose that 
    \[C = [C_1,\ldots,C_{K(K+1)}],\]
    where each $C_i$ is a nonsingular $r \times r$ matrix. Then it is enough to establish the bound
    \begin{equation}\label{MeanValueBound}
        \int_{[0,1)^{r \times K}} |f_K(\bfalp;C)|\d\bfalp \ll P^{\frac{rK(K+1)}{2} + \eps}.
    \end{equation}

    Via an application of the trivial inequality
    \begin{equation}\label{triveq}
        |a_1 \cdots a_n| \le |a_1|^n + \cdots + |a_n|^n,
    \end{equation}
    we see that
    \[\int_{[0,1)^{r \times K}} |f_K(\bfalp;C)|\d\bfalp \ll \max_{1 \le l \le K(K+1)} \Phi_l\]
    where
    \[\Phi_l = \int_{[0,1)^{r \times K}} |f_K(\bfalp;C_l)|^{K(K+1)} \d \bfalp.\]
    For a fixed $l$, the value of the integral is, by orthogonality, bounded above by the number of integer solutions of the system
    \[C_l \left[\begin{matrix}
        \Delta_{k}(\bfx_1,\bfy_1) \\
        \vdots \\
        \Delta_{k}(\bfx_r,\bfy_r)
    \end{matrix} \right] = \bf0,\]
    where
    \[\Delta_{k}(\bfx,\bfy) = \sum_{n = 1}^{K(K+1)/2} x_n^k - y^k_n,\]
    and $|\bfx_j|,|\bfy_j| \le P$.
    Since $C_l$ is nonsingular this implies that
    \[\Delta_{k}(\bfx_j,\bfy_j) = 0,\quad 1 \le j \le r,\quad 1 \le i \le r, \quad 1 \le k \le K.\]

    This is simply $J_{K(K+1)/2,K}(P)^{r}$, where 
    \[J_{s,k}(X)=\int_{[0,1)^k}\left|\sum_{|x| \leq X} e\left(\left(\alpha_1 x+\cdots+\alpha_k x^k\right)\right)\right|^{2 s} \d \alpha.\]
    By the resolution of Vinogradov's mean value theorem (see \cite{Bourgain2016, Wooley2019}) we conclude that
    \[\max_{1 \le l \le K(K+1)} \Phi_l \ll P^{\frac{rK(K+1)}{2} + \eps} .\]
    Thus we have established (\ref{MeanValueBound}).
\end{proof}
Note that since $C$ dominates the function $F(x)$ from (\ref{DomFunc}) then by Lemma \ref{LemmaWeylTypeBound} 
there exists $\sigma > 0$ and a column index $j_0 \in [s]$ for which
\[ \sup_{\bfalp \in \grm(Q)} |f_K(\bfalp;\bfc_{j_0})| \ll PQ^{-\sigma }.\]
Then $C_{[s]\backslash\{j_0\}}$ satisfies the conditions of Lemma \ref{LemmaMeanValueBound}. Thus by an application of H\"{o}lder's inequality one obtains that
\[\int_{\grm(Q)}f_K(\bfalp;C) \d \bfalp \ll P^{s-\frac{rK(K+1)}{2} + \eps}Q^{-\sigma}.\]
Upon setting $Q = P^{\delta}$ for some $ 0 < \delta < \frac{1}{3+2rK}$ we establish via Lemma \ref{LemmaMajorMinorDecomp} that
\begin{equation}\label{BndedMinorArc}
    R_K(P;C) = \grS_K(P^\delta;C)J_K(P^\delta,P;C) + o \left(P^{s-\frac{rK(K+1)}{2}} \right).
\end{equation}
\subsection{The Singular Series}
We proceed by establishing the absolute convergence of the complete singular series
\[\grS_K(C) := \sum_{q = 1}^{\infty} q^{-s}\sum_{\substack{1 \le \bfa \le q \\ (q,\bfa) = 1}}S_K(q,\bfa; C).\]
We begin with a definition. For any $r \times s$ matrix $C$ we let
\[A_K(q;C) = q^{-s}\sum_{\substack{1 \le \bfa \le q \\ (q,\bfa) = 1}}S_K(q,\bfa; C).\]
It is then immediately clear that one has absolute convergence of $\grS_K(C)$ as soon as one establishes a bound of the form $A_K(q;C) \ll q^{-(1+\sigma)+\eps}$ for some $\sigma > 0.$ We accomplish this by making use of the property that $C$ contains a partitionable $r \times rn$ matrix for $n = \floor{s/r} \ge K(K+1)$. In advance of the next lemma we recall the function $F(x)$ defined in (\ref{DomFunc}).
\begin{lemma}\label{LemmaSingSerBound}
    Let $K \ge 2$ and $C$ be a $r \times s$ matrix with $s > rK(K+1)$. If $C$ dominates $F(x)$ then there exists $\sigma > 0$ for which one has the bound
    \[A_K(q;C) \ll q^{-1-\sigma + \eps}.\]
\end{lemma}
\begin{proof}
    We begin by defining a function which will be useful later. Let $\bfb \in \Z^{K \times s}$ be given, then we define
    \[S_K^*(q,\bfb) = \prod_{1 \le j \le s} \sum_{1 \le u \le q} e\left( \sum_{1 \le k \le K} b_{k,j}u^k \right).\]
    Given this definition one may note that the following holds
    \[S_K(q,\bfa;C) = S_K^*(q,\bfa^T C).\]
    
    Via an application of Lemma \ref{LemmaWeylTypeBound} to the minor arcs $\grm(q)$, we see that there exists $j_0 \in [s]$ and $\sigma > 0$ for which
    \[\max_{\substack{1 \le \bfa \le q \\ (q,\bfa) = 1}}|S_K(q,\bfa ;C_{\{j_0\}})| \ll q^{1-\sigma}.\]
    Setting $J_0 = [s]\backslash \{j_0\}$ we obtain
    \begin{equation}\label{SingWeylTypeBound}
        A_K(q;C) \ll q^{-\sigma} A_K(q,C_{J_0}).
    \end{equation}
    Then since $C_{J_0}$ contains a partitionable $r \times rK(K+1)$ matrix one may without loss of generality write 
    \[C_{J_0} = [C_1,\ldots,C_{K(K+1)},E],\]
    where each $C_l$ is a nonsingular $r \times r$ matrix and $E$ is a $r \times (s-1-rK(K+1))$ matrix. Via an application of H\"{o}lder's inequality and the trivial inequality (\ref{triveq}) there exists $1 \le l_0 \le K(K+1)$ for which we have the asymptotic bound
    \begin{align*}
        A_K(q;C_{J_0}) &\ll q^{-rK(K+1)} \sum_{\substack{1 \le \bfa \le q \\ (q,\bfa) = 1}}|S_K(q,\bfa;C_{l_0})|^{K(K+1)} \\
        &= q^{-rK(K+1)} \sum_{\substack{1 \le \bfa \le q \\ (q,\bfa) = 1}}|S_K^*(q,\bfa^T C_{l_0})|^{K(K+1)} .
    \end{align*}

    Since the matrix $C_{l_0}$ is invertible one has that the condition $(q,\bfa) = 1$ implies $(q,\bfa^T C_{l_0}) = O(1)$. Additionally since the matrix $C$ here is fixed one similarly has that the elements of $\bfa^T C_{l_0}$ are $O(q)$. 
    Thus there exists constants $c_0$ and $c_1$ depending on at most $C_{J_0}$ for which the following inequality holds
    \[A_K(q;C_{J_0}) \ll q^{-rK(K+1)} \sum_{\substack{ |\bfb| \le c_0 q \\ (q,\bfb) \le  c_1}} |S_K^*(q,\bfb)|^{K(K+1)}, \]
    where $\bfb = [\bfb_1,\ldots,\bfb_r]$ are $K \times r$ matrices. Applying \cite[Theorem 7.1]{Vaughan1997} one obtains has the bound 
    \[S_K^*(q,\bfb) \ll \left(\prod_{1 \le j \le r}(q,\bfb_j)\right)^{1/K} q^{r(1-1/K) + \eps}.\]
    Thus we conclude that
    \[A_K(q;C_{J_0}) \ll q^{-r(K+1) + \epsilon} \sum_{\substack{|\bfb| \le c_0 q \\ (q,\bfb) \le  c_1}}\left(\prod_{1 \le j \le r}(q,\bfb_j)\right)^{K+1}. \]
    Since $A_K(q;C_{J_0})$ is multiplicative in $q$ we need only establish the bound for the case when $q$ is a prime power $p^h$. 

    Given an $r$-dimensional vector $\bfe$ of positive integers we let $\Phi(p^h;\bfe)$ denote the quantity of $K \times r$ matrices, $\bfb$, satisfying 
    \begin{equation}\label{MultConds}
        |\bfb| \le c_0 p^h \quad \text{and} \quad (p^h,\bfb_j) = p^{e_j}.
    \end{equation}
    Note that for any $\bfb$ counted by $\Phi(p^h;\bfe)$ we have that 
    \[\quad (p^h,\bfb)  = p^{\min \bfe},\]
    thus we have the equality
    \[\sum_{\substack{|\bfb| \le c_0 p^h \\ (p^h,\bfb) \le  c_1}}\left(\prod_{1 \le j \le r}(p^h,\bfb_j)\right)^{K+1} = \sum_{\substack{0 \le \bfe \le h \\ \min{\bfe} \le \log_p(c_1)}}p^{(K+1) \|\bfe\|_1} \Phi(p^h;\bfe).\]
    
    Note that for any given choice of $0 \le e_j \le h$ one has that there at most $O(p^{K(h-e_j)})$ valid choices for $\bfb_j$ satisfying (\ref{MultConds}). One then concludes that
    \[\Phi(p^h;\bfe) \ll p^{K(rh-\|\bfe\|_1)}.\]
    Hence
    \begin{align*}
        A_K(p^h;C_{J_0}) &\ll p^{-h(r - \eps)} \sum_{\substack{0 \le \bfe \le h \\ \min{\bfe} \le \log_p(c_1)}} p^{\|\bfe\|_1} \ll p^{h(-r +(r-1)+ \eps)} = p^{-h(1 - \eps)}.
    \end{align*}
    Thus we establish that $A_K(q;C_{J_0})\ll q^{-1+\eps}$. Combining this with (\ref{SingWeylTypeBound}) we establish the lemma.
\end{proof}
We conclude that $\grS_K$ converges absolutely and there exists $\sigma > 0$ for which one has
\[|\grS_K - \grS_K(Q)| \ll  Q^{-\sigma+\eps}.\]

\subsection{The Singular Integral}

Here we will show that the complete singular integral
\begin{equation}\label{CompSingInt}
    J_K(P;C) := \lim_{Q \to \infty} \int_{\substack{\bfgam \in \R^{r \times K} \\ |\bfgam_{k}| \le QP^{-k}}} I_K(P,\bfgam;C) \d\bfgam
\end{equation}
is absolutely convergent. We begin with a slight simplification, note that via a change of variables we obtain
\[J_K(P;C) = P^{s- \frac{rK(K+1)}{2}} J_K(1;C).\]
Thus, to prove (\ref{CompSingInt}) converges absolutely it suffices to show that the integral
\[J_K(1;C) = \lim_{Q \to \infty} \int_{|\bfgam| \le Q} I_K(1,\bfgam;C) \d \bfgam,\]
converges absolutely.
\begin{lemma}\label{LemmaSingIntBound}
    Let $C$ be an $r \times s$ with $s>rK(K+1)$ and suppose $C$ dominates
    \[\frac{x- \rem (s,r)}{\floor{\frac{s}{r}}}.\]
    Then one has that
    \[\lim_{Q \to \infty} \int_{|\bfgam| \le Q} I_K(1,\bfgam;C) \d \bfgam\]
    converges absolutely.
\end{lemma}
\begin{proof}
    Since $C$ contains a partitionable submatrix of size $r \times rK(K+1)$, we may use the trivial estimate over oscillatory integrals and suppose that $C = [C_1,\ldots,C_{K(K+1)}]$. We begin by first defining a useful function. For any 
    \[\bfxi = [\bfxi_1,\ldots, \bfxi_r] \in \R^{K \times r} \]
    we let
    \[I_K^*(\bfxi) = \prod_{1 \le j \le r} \int_{-1}^{1} e \left( \sum_{1 \le k \le K} \xi_{k,j} z^k \right) \d z.\]
    Given this definition one may note that the following holds
    \[I_K(1,\bfgam;C) = I_K^*(\bfgam^T C).\]

    Making use of (\ref{triveq}), there exists $1 \le l_0 \le K(K+1)$ such that
    \begin{align*}
        \int_{|\bfgam| \le Q} |I_K(1,\bfgam;C)| \d \bfgam &\ll \int_{|\bfgam| \le Q} |I_K(1,\bfgam;C_{l_0})|^{K(K+1)} \d \bfgam \\ 
        &= \int_{|\bfgam| \le Q} |I_K^*(\bfgam^T C_{l_0})|^{K(K+1)} \d \bfgam.
    \end{align*}
    Since $C_{l_0}$ is nonsingular we can make a nonsingular linear change of variables  $\bfxi =  \bfgam^T C_{l_0}$. Additionally, from the argument in \cite[Theorem 7.3]{Vaughan1997}, we deduce that
    \begin{equation}\label{IBoundProd}
        I_K^*(\bfxi) \ll \prod_{1 \le j \le r} \min \{1,|\bfxi_{j}|^{-1/K}\}.
    \end{equation}
    Hence
    \begin{align*}
        \int_{|\bfgam| \le Q} |I_K^*(\bfgam^T C_{l_0})|^{K(K+1)} \d \bfgam &\ll \int_{|\bfxi| \le Q}|I_K^*(\bfxi)|^{K(K+1)} \d \bfxi \\
        &\ll \prod_{1 \le j \le r} \int_{|\bfxi_j| \le Q} \min\{1, |\bfxi_j|^{-K-1)} \}\d \bfxi_j \\
        &= \left( \int_{\substack{\bfx \in \R^K \\ |\bfx| \le Q}} \min\{1, |\bfx|^{-K-1}\} \d \bfx \right)^r.
    \end{align*}
    Thus it is enough to show that the integral
    \[\lim_{Q \to \infty}\int_{\substack{\bfx \in \R^K \\ |\bfx| \le Q}} \min\{1, |\bfx|^{-K-1}\} \d \bfx\]
    converges absolutely. This follows immediately by elementary analysis as one sees that
    \[\int_{\substack{\bfx \in \R^K \\ |\bfx| \ge Q}} \min\{1, |\bfx|^{-K-1}\} \d \bfx \ll Q^{-1}. \]
\end{proof}
Upon combining the results of Lemma \ref{LemmaSingSerBound} and Lemma \ref{LemmaSingIntBound} with (\ref{BndedMinorArc}) we conclude that
\[R_K(P;C) = P^{s-\frac{rK(K+1)}{2}}\left( \sigma_K(C) + o(1)\right),\]
where $\sigma_K(C) = \grS_K(C) J_K(1;C)$. Since we have shown the absolute convergence of both the singular series and singular integral it follows from the arguments of \cite{Brandes2021} that $\sigma_K(C) > 0$ whenever there exists nonsingular real and $p$-adic solutions to (\ref{GenSys}). With this, we have established Theorem \ref{CircMethThm}. 

\section{Analyzing the $K$-multimagic square system}\label{sec:multmagicsquare}
Let $N \ge 4$ and recall from section \ref{sec:overview} the quantity $M_{N,K}(P)$. We now establish the existence of a matrix 
\[C^{\text{magic}}_{N} \in \{-1,0,1\}^{2N \times N^2}\] 
for which $M_{N,K}(P) = R_K(P;C^{\text{magic}}_{N})$. Note that a matrix $\bfZ = (z_{i,j})_{\substack{1 \le i,j \le N}}$ is a \textbf{MMS}$(K,N)$ if and only if for all $1 \le k \le K$ it satisfies the simultaneous conditions
\begin{equation}\label{MSSystem1}
    \sum_{1 \le i \le N} z_{i,j}^{k} = \sum_{1 \le i \le N} z_{i,i}^{k} \quad \text{ for }\quad  1 \le j \le N,
\end{equation}
\begin{equation}\label{MSSystem2}
    \sum_{1 \le j \le N} z_{i,j}^{k} = \sum_{1 \le j \le N} z_{j,N-j+1}^{k} \quad \text{ for }\quad  1 \le i \le N.
\end{equation}
One may wonder if these equations are equivalent to those of a \textbf{MMS}$(K,N)$, indeed it does not seem clear that the main diagonal and anti-diagonal are equal at first glance. One can show that this is implied by the above by simply summing over all $j$ in (\ref{MSSystem1}) and noting that this is equal to summing over all $i$ in (\ref{MSSystem2}). Upon dividing out a factor of $N$ one deduces that (\ref{MSSystem1}) and (\ref{MSSystem2}) imply
\[\sum_{1 \le i \le N} z_{i,i}^{k} = \sum_{1 \le j \le N} z_{j,N-j+1}^{k}.\]
Before we construct a matrix corresponding to this system we must first establish some notational shorthand. Let ${\bf1}_n$, or respectively ${\bf0}_n$, denote a $n$-dimensional vector of all ones, or respectively all zeros. Let $\bfe_n(m)$ denote the $m$th standard basis vector of dimension $n$. For a fixed $N$ we define 
\[D_1(N) = \{(i,j) \in [N]^2: i = j\},\]
and
\[D_2(N) = \{(i,j) \in [N]^2: i +j = N+1\}.\]
For each $(i,j) \in [N]^2$ we define the $2N$-dimensional vectors
\[\bfd_{i,j} = 
\begin{cases}
(\bfe_N(i)-{\bf1_N},\bfe_N(j) - {\bf1_N})  & (i,j) \in D_1(N) \cap D_2(N) \\
(\bfe_N(i)-{\bf1_N},\bfe_N(j))  & (i,j) \in D_1(N) \backslash D_2(N) \\
(\bfe_N(i),\bfe_N(j) - {\bf1_N})  & (i,j) \in D_2(N) \backslash D_1(N) \\
(\bfe_N(i),\bfe_N(j))  & \text{otherwise} \\
\end{cases}\]
Let $\phi: [N^2] \to [N]^2$ be any fixed bijection, then the $2N \times N^2$ matrix
\[C^{\text{magic}}_{N} = C^{\text{magic}}_{N}(\phi) = [\bfd_{\phi(1)},\ldots,\bfd_{\phi(N^2)}],\]
corresponds to the system defined by (\ref{MSSystem1}) and (\ref{MSSystem2}) up to some arbitrary relabeling of variables defined by the bijection $\phi$. For any subset $J \subset [N]^2$ we define $(C^{\text{magic}}_{N})_J := (C^{\text{magic}}_{N})_{\phi^{-1}(J)}$ in the sense defined in section \ref{sec:intro}.

Before we move on it will be useful to define a notion of equivalence between subsets of columns of a matrix. Let $M$ be a $r \times s$ matrix, then we say $J_1,J_2 \subset [s]$ are $M$-isomorphic if 
    \[\im (M)_{J_1} = \im (M)_{J_2}.\]

\begin{lemma}\label{RankGrowthLemma}
    Let $N \ge 4$, then for any $J \subset [N]^2$ we have that
    \[\rank (C^{\text{magic}}_{N})_J \ge \begin{cases}
        \min\left\{\ceil{2\sqrt{|J|}}-1, 2N-2\right\}+ E_N(J) & \text{if }|J| \neq (N-1)^2+1, \\
        2N-3 + E_N(J) & \text{if } |J| = (N-1)^2+1,
    \end{cases}\]
    where 
    \[E_N(J) = \dim\left( \im (C^{\text{magic}}_{N})_J \cap \im \left[\begin{matrix}
        {\bf1}_N & {\bf0}_N \\
        {\bf0}_N & {\bf1}_N
    \end{matrix}\right] \right).\]
\end{lemma}
\begin{proof}
    Let 
    \[B = \left[\begin{matrix}
        {\bf1}_N & {\bf0}_N \\
        {\bf0}_N & {\bf1}_N
    \end{matrix}\right],\]
    then by elementary linear algebra, one has the equality
    \[\rank (C^{\text{magic}}_{N})_J + \rank B = \rank [(C^{\text{magic}}_{N})_J,B] + \dim\left( \im (C^{\text{magic}}_{N})_J \cap \im B \right).\]
    Thus we deduce
    \[\rank (C^{\text{magic}}_{N})_J = \rank [(C^{\text{magic}}_{N})_J,B] - 2 + E_N(J).\]
    Thus it suffices to understand the rank of the matrix $[(C^{\text{magic}}_{N})_J,B]$ given $J \subset [N]^2$. Via rank preserving elementary column operations we have that
    \[\im [(C^{\text{magic}}_{N})_J,B] = \im [A_J,B],\]
    where
    \[A_J = 
    \left[\begin{matrix}
        \bfe_{N}(i) \\
        \bfe_{N}(j)
    \end{matrix} \right]_{(i,j) \in J}.
    \]
    This simplification motivates our choice of $B$. Then via the same elementary linear algebra identity used above we further see that
    \[\rank (C^{\text{magic}}_{N})_J = \rank A_J - \dim(\im A_J \cap \im B) + E_N(J).\]
    Here, an important idea that must be considered is that of \emph{projections}, \emph{equivalent sets}, and \emph{irreducible} sets . Given $J \subset [N]^2$ we define 
    \[\pi_1(J) = \{i: (i,j) \in J\}\quand \pi_2(J) = \{j: (i,j) \in J\}.\]
    We call two sets $J_1,J_2$ equivalent if there exists bijections $\phi_1,\phi_2 : [N] \to [N]$ such that either
    \[J_1 = \{(\phi_1(i),\phi_2(j)): (i,j) \in J_2\},\]
    or
    \[J_1 = \{(\phi_1(j),\phi_2(i)): (i,j) \in J_2\}.\]
    Clearly if two sets $J_1,J_2$ are equivalent then $\rank A_{J_1} = \rank A_{J_2}$. We call a set $J \subset [N]^2$ irreducible if whenever $J = A \cup B$, where $A$ and $B$ are nonempty, one either has $\pi_1(A) \cap \pi_1(B) \neq \emptyset$ or $\pi_2(A) \cap \pi_2(B) \neq \emptyset$. Given this definition we see that if $J_1$ is irreducible and $J_2$ is equivalent to $J_1$ then $J_2$ is irreducible. It is also not hard to see that if $J,J' \subset [N]^2$ are sets satisfying
    \[\pi_1(J) \cap \pi_1(J') = \emptyset \quand \pi_2(J) \cap \pi_2(J') = \emptyset,\]
    then $\rank A_{J \cup J'} = \rank A_{J}+\rank A_{J'}$. Thus given an arbitrary set $J \subset [N]^2$, there exists a unique number $l= l(J)$ for which
    \[J = \bigsqcup_{1 \le k \le l} J_k,\]
    where each $J_k$ is irreducible and satisfies
    \[\pi_1(J_i) \cap \pi_1(J_j) = \emptyset \quand \pi_2(J_i) \cap \pi_2(J_j) = \emptyset \text{ for }i \neq j.\]
    Thus given our above observations we deduce that
    \[\rank A_J = \sum_{1 \le k \le l}\rank A_{J_k}.\]
    We additionally note the trivial observation that
    \begin{equation}\label{fudgefactorident}
        \dim(\im A_J \cap \im B)=1 \iff \#\pi_1(J)=\#\pi_2(J)=N.
    \end{equation}
    From the fundamental relation amongst the columns of $A_{[N]^2}$, namely
    \begin{equation}\label{FundRelation}
        \left[\begin{matrix}
        \bfe_{N}(i_1) \\
        \bfe_{N}(j_1)
    \end{matrix} \right] - \left[\begin{matrix}
        \bfe_{N}(i_1) \\
        \bfe_{N}(j_2)
    \end{matrix} \right] + \left[\begin{matrix}
        \bfe_{N}(i_2) \\
        \bfe_{N}(j_2)
    \end{matrix} \right] - \left[\begin{matrix}
        \bfe_{N}(i_2) \\
        \bfe_{N}(j_1)
    \end{matrix} \right] = {\bf0}_{2N},
    \end{equation}
    for irreducible $J$ we have that
    \begin{equation}\label{irrrankident}
        \rank A_{J} = \#\pi_1(J)+\#\pi_2(J)-1.
    \end{equation}
    Thus we reduce to the geometric problem of finding a lower bound for $\#\pi_1(J)+\#\pi_2(J)-1$ in terms of $|J|$. Note that if $(n-1)^2 < |J| \le n(n-1)$ then it is not hard to see that 
    \[\#\pi_1(J)+\#\pi_2(J)-1 \ge 2n-2,\] 
    similarly if $n(n-1) < |J| \le n^2$ then 
    \[\rank A_{J} = \#\pi_1(J)+\#\pi_2(J)-1 \ge 2n-1.\] 
    It is not hard to verify that this implies
    \[\rank A_{J} \ge 2 \sqrt{J}-1.\]
    Thus we deduce
    \[\rank A_J \ge  \sum_{1 \le k \le l}2 \sqrt{|J_k|}-1, \quad |J| = \sum_{1 \le k \le l}|J_k|.\]
    Via some optimization with the given constraints and noting that $|J_k| \ge 1$, one may establish that
    \[\rank A_J \ge 2\sqrt{|J|-l+1}+l-2 \Rightarrow \rank A_J \ge \ceil{2\sqrt{|J|-l+1}}+l-2.\]
    Note that for $|J| > 3$ and $l \ge 2$ one has that
    \[\ceil{2\sqrt{|J|}}-1 \le \ceil{2\sqrt{|J|-l+1}}+l-2,\]
    with equality only when $l = 2$ and $|J| = (n-1)^2+1$ or $|J| = n(n-1)+1$ for some $3 \le n \le N$. Manually checking the cases when $1 \le |J| \le 3$ we deduce the bound
    \[\rank A_J \ge \ceil{2\sqrt{|J|}}-1 \quad \text{ for all } J \subset [N]^2.\]

    Given this lower bound on the rank of $A_J$ it is now worth determining the sets $J$ which achieve this optimal lowest rank, luckily, from the information we have derived thus far this is not a difficult task. From now on, we call a set $J$ \emph{optimal} if one has that
    \[\rank A_J =  \ceil{2\sqrt{|J|}}-1.\]
    Note that if $J$ is an optimal irreducible set then from (\ref{fudgefactorident}) and (\ref{irrrankident}) one must have $\rank A_J = 2N-1$ in order for $\dim(\im A_J \cap \im B)=1$. This allows us to conclude that
    \[\rank A_J - \dim(\im A_J \cap \im B) \ge \min \left\{\ceil{2 \sqrt{J}}-1, 2N-2 \right\},\]
    when $J$ is an optimal irreducible set.

    From our above analysis there are only three cases in which an optimal set may not be irreducible, either $1 \le |J| \le 3$, $|J| = (n-1)^2+1$, or $|J| = n(n-1)+1$ for some $3 \le n \le N$. We must now establish when these optimal sets may satisfy
    \[\dim(\im A_J \cap \im B)=1.\]
    
    If $1 \le |J| \le 3$, then because we are assuming $N \ge 4$ we have from (\ref{fudgefactorident}) that $\dim(\im A_J \cap \im B)=0$. If $|J| = (n-1)^2+1$ then it is not difficult to see that every optimal nonirreducible set of this size is equivalent to the set
    \[[n-1]^2 \sqcup P\]
    where $P \in ([n,N] \cap \Z)^2$. From (\ref{fudgefactorident}) we see that $\dim(\im A_J \cap \im B)=1$ if and only if $J$ is equivalent to
    \[[N-1]^2 \sqcup \{N,N\} \Rightarrow |J| = (N-1)^2+1,\]
    in this case we see that 
    \[\rank A_J - \dim(\im A_J \cap \im B) \ge 2N-3.\]
    Finally, if $|J| = n(n-1)+1$ then from the above one has $\dim(\im A_J \cap \im B)=1$ if and only if $J$ is equivalent to
    \[[N] \times [N-1] \cup \{1,N\},\]
    which is irreducible, thus $J$ must be irreducible in this case and our previous bound on the rank holds. We end by noting that because $0 \le \dim(\im A_J \cap \im B) \le 1$, the lower bounds we have established for optimal sets hold for all sets.
\end{proof}
We now establish our final necessary result.
\begin{lemma}\label{CombinatorialLemma}
    Let $N \ge 4$ and $J \subset [N]^2$. If $|J| > N(N-1)$ then $E_N(J) = 2$.
\end{lemma}
\begin{proof}
    If $N = 4$ then one may through brute force computation establish that for every $J \subset [4]^2$ of size $13$ one has $\rank (D_4)_J = 8$ thus $E_4(J) = 2$. Henceforth, we will be working under the assumption that $N \ge 5$. We also take 
    \[\chi_2(N) = \begin{cases}
        1, & \text{when}\, 2 \nmid N, \\
        0, & \text{when}\, 2|N.
    \end{cases}\]
    Let $J \subset [N]^2$ be of size $N(N-1)+1$, and define $S = J \backslash (D_1(N) \cup D_2(N))$ and set $S^c:= J \backslash S$. Due to the size of $J$ it must be the case that $|S| \ge N+1+\chi_2(N)$. We now split into subcases depending on the rank of $(C^{\text{magic}}_{N})_{S^c}$. 
    
    If $\rank (C^{\text{magic}}_{N})_{S^c} = 2N-1$ then $\im (C^{\text{magic}}_{N})_{S^c} = \im A_{[N]^2}$ where $A$ is defined in the proof of Lemma \ref{RankGrowthLemma}. It is not hard to see that every $\bfx \in \im A_{[N]^2}$ satisfies 
    \[\bfx \cdot \left[\begin{matrix}
        {\bf1}_N \\
        -{\bf1}_N
    \end{matrix} \right] = {\bf0}_{2N}.\]
    It is also clear to see that at least one of the elements of $S$ does not lie in $D_1(N) \cap D_2(N)$, say $(i_0,j_0)$. Then one may check that
    \[ \left|\bfd_{i_0,j_0} \cdot  \left[\begin{matrix}
        {\bf1}_N \\
        -{\bf1}_N
    \end{matrix} \right] \right| = N \Rightarrow \bfd_{i_0,j_0} \notin \im A_{[N]^2},\]
    whence $\rank (C^{\text{magic}}_{N})_J > \rank (C^{\text{magic}}_{N})_{S^c} = 2N-1$ and trivially $E_N(J) = 2$. 
    
    Suppose $\rank (C^{\text{magic}}_{N})_{S^c} = 2N-2$, then by our previous analysis on the matrix $A$ it follows that $S^c$ is without loss of generality $C^{\text{magic}}_{N}$-isomorphic to the set 
    \[\overline{S^c} = \{(i,j) \in [N]^2 \backslash (D_1(N) \cup D_2(N)): i \neq N\},\]
    thus $|S^c| \le N^2-3N+2+\chi_2(N)$, whence $|S| \ge 2N-1-\chi_2(N)$. 
    
    If $S$ contains one element in $D_1(N)$, say $(i_1,j_1)$, and another in $D_2(N)$, say $(i_2,j_2)$ neither of which belonging to row $N$ then we are done. This is because there exists $1 \le k \le N$ and $l \neq N$ such that
    \begin{equation*}
        \bfd_{i_1,j_1} - \left[\begin{matrix}
        e_N(i_2) \\
        e_N(k)
    \end{matrix} \right] + \left[\begin{matrix}
        e_N(l) \\
        e_N(k)
    \end{matrix} \right] - \left[\begin{matrix}
        e_N(l) \\
        e_N(j_0)
    \end{matrix} \right] = -\left[\begin{matrix}
        {\bf1}_N \\
        {\bf0}_N
    \end{matrix} \right].
    \end{equation*}
    Because the 3 latter vectors in the left hand side are in $\im (C^{\text{magic}}_{N})_{S^c}$ we obtain 
    \[\left[\begin{matrix}
        {\bf1}_N \\
        {\bf0}_N
    \end{matrix} \right] \in \im (C^{\text{magic}}_{N})_J.\]
    This exact same trick may be done with $(i_2,j_2)$ to obtain
    \[\left[\begin{matrix}
        {\bf0}_N \\
        {\bf1}_N
    \end{matrix} \right] \in \im (C^{\text{magic}}_{N})_J.\]
    Note that the above definitely happens if $2N-1-\chi_2(N) \ge N+2+\chi_2(N)$, which is always true for $N \ge 5$, hence $E_N(J) = 2$.
    
    We now consider the case in which $\rank (C^{\text{magic}}_{N})_{S^c} \le 2N-3$. By our previous analysis on the matrix $A$ we see that the largest set $S^c \subset [N]^2 \backslash (D_1(N) \cup D_2(N))$ for which $\rank (C^{\text{magic}}_{N})_{S^c} \le 2N-3$ has size at most $|S^c| \le N^2 - 4N+5+\chi_2(N)$, this implies that $|S| \ge 3N-4 - \chi_2(N)$. Note that since $|S| \le 2N - \chi_2(N)$ this is impossible for $N \ge 5$. 
\end{proof}
Combining Lemma \ref{RankGrowthLemma} and Lemma \ref{CombinatorialLemma} and noting that the for any $1 \le m < N^2$ one must have that
\[0 \le \min_{|J| = m+1}\rank(C^{\text{magic}}_{N})_J - \min_{|J| = m}\rank(C^{\text{magic}}_{N})_J \le 1,\]
we conclude that for $|J| \neq (N-1)^2+1$ one has
\[\rank(C^{\text{magic}}_{N})_J \ge \begin{cases}
    \ceil{2\sqrt{|J|}}-1 & 1 \le |J| \le N(N-1)-1, \\
    |J| - N^2+3N-1 & N(N-1)-1 \le |J| \le N(N-1)+1, \\
    2N & N(N-1)+1 \le |J| \le N^2,
\end{cases}\]
and
\[\rank(C^{\text{magic}}_{N})_J \ge 2N-3,\]
when $|J| = (N-1)^2+1$.
Thus one may deduce that the matrix $C^{\text{magic}}_{N}$ dominates $F(x)$ from (\ref{DomFunc}) and so we have established the quantitative Hasse principle for \textbf{MMS}$(K,N)$ for all $N > 2K(K+1)$. We now go further and focus on establishing the existence of nonsingular local solutions.

\section{Existence of nonsingular local $K$-multimagic squares}\label{sec:nonsingsoln}
We begin by introducing some notation, for any $s$-dimensional vector $\bfxi$, we define $\text{diag}(\bfxi)$ to be the $s \times s$ diagonal matrix which has the elements of $\bfxi$ as its diagonal entries. 

Let us now consider the Jacobian matrix associated to the equations (\ref{GenSys}) defined by the matrix $C$ evaluated at $\bfx$,
\[\bfJ(\bfx;C) = \left[\begin{matrix}
    \bfJ_1(\bfx,C) \\
    \vdots \\
    \bfJ_K(\bfx,C)
\end{matrix}  \right],\text{where}\quad \bfJ_k = k C \text{diag}(\bfx)^{k-1}.\]
If we replace $C$ with $C_J$ for any $J \subset [s]$, then one certainly has that
\[\rank\left( \bfJ(\bfx;C) \right) \ge \rank \left(\bfJ(\bfx_J;C_J) \right).\]
Thus if we wish to show the Jacobian has full rank it suffices to show that the Jacobian associated to a submatrix has full rank. This lets us reduce to the case in which the matrix $C$ is a partitionable $r \times rK$ matrix $C = [M_1,\ldots,M_{K}]$. Next we define the block diagonal matrix
\[\Tilde{C} = \left[ \begin{matrix}
    M_1^{-1} & {\bf0} & \cdots & {\bf0}  \\
    {\bf0} & M_2^{-1} & \cdots & {\bf0}  \\
    \vdots & \vdots & \ddots & \vdots  \\
    {\bf0} & {\bf0} & \cdots & M_{K}^{-1}
\end{matrix} \right].\]
Note that this matrix is clearly nonsingular, whence
\begin{align}\label{RankEQ}
    \rank \left(\bfJ(\bfx;C)\right)= \rank \left(\bfJ(\bfx;C) \Tilde{C}\right) 
\end{align}
This simplifies our problem because one has that
\[\bfJ_k(\bfx;C) \Tilde{C} = k  \left[\text{diag}(\bfx_1)^{k-1},\text{diag}(\bfx_2)^{k-1},\ldots,\text{diag}(\bfx_{K})^{k-1}\right],\]
where
\[\bfx_l = (x_{1+2N(l-1)},\ldots,x_{2Nl}).\]
Then by swapping rows and columns and utilizing (\ref{RankEQ}) one may deduce that
\[\rank \left(\bfJ(\bfx;C)\right) = \sum_{1 \le j \le 2N}\rank V'(x_{j},x_{j+2N},\ldots,x_{j+2N(K-1)}),\]
where
\[V'(\bfy) = \left[\begin{matrix}
    1 & 1 & \cdots & 1 \\
    2y_1 & 2y_2 & \cdots & 2y_{K} \\
    \vdots & \vdots & \ddots & \vdots \\
    Ky_{1}^{K-1} & Ky_{2}^{K-1} & \cdots & Ky_{K}^{K-1}
\end{matrix}  \right].\]
Comparing this matrix to a Vandermonde matrix it is not too hard to see 
\[\det(V'(\bfy)) = K! \prod_{1 \le i<j \le K} (y_i-y_j).\]
Thus we conclude that the Jacobian $\bfJ(\bfx;C)$, defined by $C$ evaluated at $\bfx$ has full rank if there exists $K$ many disjoint ordered subsets 
\[J_l = \{j(l,1),\ldots,j(l,r)\} \subset [s]\] 
for $1 \le l \le K$ which satisfy
\[\rank C_{J_l} = r \text{ for all } 1 \le l \le K\]
and 
\[x_{j(l_1,n)} \neq x_{j(l_2,n)} \text{ for all } 1 \le l_1,l_2 \le K \text{ and } 1 \le n \le r.\]

By \cite[Lemma 1]{Low1988} and Lemma \ref{RankGrowthLemma} one has that $C^{\text{magic}}_{N}$ contains a $2N \times 2NK$ partitionable submatrix because $K \le \floor{N/2}$. Suppose $\bfZ = (z_{i,j})_{1 \le i,j \le N}$ is a \textbf{MMS}$(K,N)$, hence it satisfies
\[\left[ \bfd_{\phi(1)},\ldots,\bfd_{\phi(N^2)} \right] \left[ \begin{matrix}
    z_{\phi(1)} \\
    \vdots \\
    z_{\phi(N^2)}
\end{matrix} \right] = {\bf0}_{N^2},\]
where $\phi$ is any bijection from $[N^2]$ to $[N]^2$. Thus if we establish the existence of a bijection $\phi: [N^2] \to [N]^2$ for which one has
\begin{equation}\label{MaxRank}
    \rank \left[ \bfd_{\phi(1+2N(l-1))},\bfd_{\phi(2+2N(l-1))},\ldots,\bfd_{\phi(2Nl)} \right] = 2N \quad \text{for all} \quad 1 \le l \le K,
\end{equation}
and
\begin{equation}\label{ResCond}
  z_{\phi(n+2N(l_1-1))} \neq z_{\phi(n+2N(l_2-1))} \text{ for all } 1 \le l_1,l_2 \le K \text{ and }1 \le n \le 2N,  
\end{equation}
then by the above work we may conclude that the Jacobian associated to the $K$-multimagic square system evaluated at $\bfZ$ has full rank.

First we must fix suitable disjoint subsets $J_l(N) \subset [N]^2$ of size $2N$ for $1 \le l \le \floor{N/2}$ satisfying
\[\rank \left[ \bfd_{i,j} \right]_{(i,j) \in J_l(N)} = 2N.\]
The explicit definitions of these partitions may not seem straightforward but we will also provide a figure to hopefully illuminate what these partitions look like. Given a fixed $N \ge 4$ we first define our partition in the case where $N$ is even. For every $1 \le l \le \floor{N/2}$ let
\begin{equation*}
    J_l(N) = J^{(1)}_{l}(N) \sqcup J^{(2)}_{l}(N),
\end{equation*}
instead of writing the definition of these sets $J^{(1)}_{l}(N)$ and $J^{(2)}_{l}(N)$ explicitly, we will instead define conditions under which we may determine that a pair $(i,j)$ belong to $J^{(1)}_{l}(N)$ or $J^{(2)}_{l}(N)$. If $N$ is even then
\begin{equation*}
    (i,j) \in J^{(1)}_{l}(N) \iff i-j \equiv 2(l-1) \nmod N,
\end{equation*}
and
\begin{equation*}
    (i,j) \in J^{(2)}_{l}(N) \iff i-j \equiv 2l-1 \nmod N.
\end{equation*}
For the odd case we split into further subcases depending on parity of $\frac{N+1}{2}$. If 
\[j \notin \left[2-\rem\left(\frac{N+1}{2},2\right),\frac{N+3}{2}\right]\] 
then
\[(i,j) \in J^{(1)}_{l}(N) \iff i-j \equiv 2(l-1) \nmod N,\]
if 
\[j \in \left[2-\rem\left(\frac{N+1}{2},2\right),\frac{N+3}{2}\right]\quad \text{and}\quad j \equiv \frac{N+3}{2} \nmod 2\]
then
\[(i,j) \in J^{(1)}_{l}(N) \iff i-j+1 \equiv 2(l-1) \nmod N,\]
and finally if 
\[j \in \left[2-\rem\left(\frac{N+1}{2},2\right),\frac{N+3}{2}\right]\quad \text{and} \quad j \equiv \frac{N+1}{2} \nmod 2\] 
then
\[(i,j) \in J^{(1)}_{l}(N) \iff i-j-1 \equiv 2(l-1) \nmod N.\]
We define $J^{(2)}_{l}(N)$ for the odd case in a similar way. If 
\[j \notin \left[2-\rem\left(\frac{N+1}{2},2\right),\frac{N+3}{2}\right]\] 
then
\[(i,j) \in J^{(2)}_{l}(N) \iff i-j \equiv 2l-1 \nmod N,\]
if 
\[j \in \left[2-\rem\left(\frac{N+1}{2},2\right),\frac{N+3}{2}\right]\quad \text{and} \quad j \equiv \frac{N+3}{2} \nmod 2\] 
then
\[(i,j) \in J^{(2)}_{l}(N) \iff i-j+1 \equiv 2l-1 \nmod N,\]
and finally if 
\[j \in \left[2-\rem\left(\frac{N+1}{2},2\right),\frac{N+3}{2}\right]\quad \text{and} \quad j \equiv \frac{N+1}{2} \nmod 2\] 
then
\[(i,j) \in J^{(2)}_{l}(N) \iff i-j-1 \equiv 2l-1 \nmod N.\]
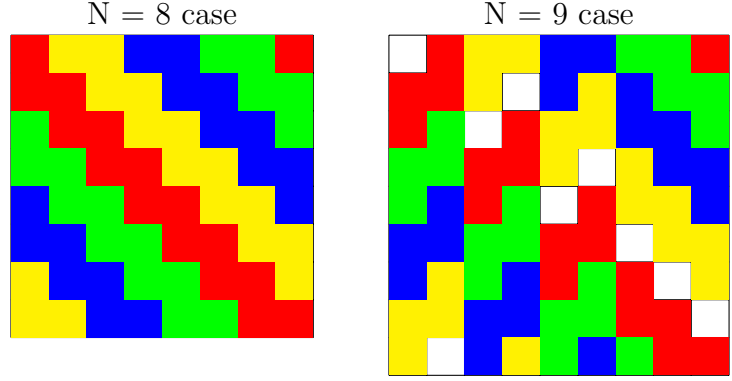
\begin{figure}[htbp]
    \centering
    \begin{tikzpicture}[scale=0.5, y=-1cm]
        \draw (0,0) grid (8,8);
        \foreach \x in {1,...,8}{
            \foreach \y in {1,...,8}{
            \pgfmathtruncatemacro\diag{Mod((\x-\y),8)}
            \ifnum\diag=0
            \fill[red!, opacity=0.3] (\y-1,\x-1) rectangle (\y,\x);
            \fi
            \ifnum\diag=1
            \fill[red!, opacity=0.3] (\y-1,\x-1) rectangle (\y,\x);
            \fi
            \ifnum\diag=2
            \fill[green!, opacity=0.3] (\y-1,\x-1) rectangle (\y,\x);
            \fi
            \ifnum\diag=3
            \fill[green!, opacity=0.3] (\y-1,\x-1) rectangle (\y,\x);
            \fi
            \ifnum\diag=4
            \fill[blue!, opacity=0.3] (\y-1,\x-1) rectangle (\y,\x);
            \fi
            \ifnum\diag=5
            \fill[blue!, opacity=0.3] (\y-1,\x-1) rectangle (\y,\x);
            \fi
            \ifnum\diag=6
            \fill[yellow!, opacity=0.3] (\y-1,\x-1) rectangle (\y,\x);
            \fi
            \ifnum\diag=7
            \fill[yellow!, opacity=0.3] (\y-1,\x-1) rectangle (\y,\x);
            \fi
            }
        }
        \node[above] at (4,0) {N = 8 case};
        
        \begin{scope}[xshift=10cm]
            \draw (0,0) grid (9,9);
        \foreach \x in {1,...,9}{
            \foreach \y in {1,...,9}{
            \pgfmathtruncatemacro\diag{Mod((\x-\y),9)}
            \pgfmathtruncatemacro\ycheck{Mod(\y,2)}
                \ifthenelse{\y<7}{
                    \ifthenelse{\ycheck=1}{
                    \ifnum\diag=0
                    \fill[red!, opacity=0.3] (\y,\x-1) rectangle (\y+1,\x);
                    \fi
                    \ifnum\diag=1
                    \fill[red!, opacity=0.3] (\y,\x-1) rectangle (\y+1,\x);
                    \fi
                    \ifnum\diag=2
                    \fill[green!, opacity=0.3] (\y,\x-1) rectangle (\y+1,\x);
                    \fi
                    \ifnum\diag=3
                    \fill[green!, opacity=0.3] (\y,\x-1) rectangle (\y+1,\x);
                    \fi
                    \ifnum\diag=4
                    \fill[blue!, opacity=0.3] (\y,\x-1) rectangle (\y+1,\x);
                    \fi
                    \ifnum\diag=5
                    \fill[blue!, opacity=0.3] (\y,\x-1) rectangle (\y+1,\x);
                    \fi
                    \ifnum\diag=6
                    \fill[yellow!, opacity=0.3] (\y,\x-1) rectangle (\y+1,\x);
                    \fi
                    \ifnum\diag=7
                    \fill[yellow!, opacity=0.3] (\y,\x-1) rectangle (\y+1,\x);
                    \fi
                    }{
                    \ifnum\diag=0
                    \fill[red!, opacity=0.3] (\y-2,\x-1) rectangle (\y-1,\x);
                    \fi
                    \ifnum\diag=1
                    \fill[red!, opacity=0.3] (\y-2,\x-1) rectangle (\y-1,\x);
                    \fi
                    \ifnum\diag=2
                    \fill[green!, opacity=0.3] (\y-2,\x-1) rectangle (\y-1,\x);
                    \fi
                    \ifnum\diag=3
                    \fill[green!, opacity=0.3] (\y-2,\x-1) rectangle (\y-1,\x);
                    \fi
                    \ifnum\diag=4
                    \fill[blue!, opacity=0.3] (\y-2,\x-1) rectangle (\y-1,\x);
                    \fi
                    \ifnum\diag=5
                    \fill[blue!, opacity=0.3] (\y-2,\x-1) rectangle (\y-1,\x);
                    \fi
                    \ifnum\diag=6
                    \fill[yellow!, opacity=0.3] (\y-2,\x-1) rectangle (\y-1,\x);
                    \fi
                    \ifnum\diag=7
                    \fill[yellow!, opacity=0.3] (\y-2,\x-1) rectangle (\y-1,\x);
                    \fi
                    }
                }{
                \ifnum\diag=0
                \fill[red!, opacity=0.3] (\y-1,\x-1) rectangle (\y,\x);
                \fi
                \ifnum\diag=1
                \fill[red!, opacity=0.3] (\y-1,\x-1) rectangle (\y,\x);
                \fi
                \ifnum\diag=2
                \fill[green!, opacity=0.3] (\y-1,\x-1) rectangle (\y,\x);
                \fi
                \ifnum\diag=3
                \fill[green!, opacity=0.3] (\y-1,\x-1) rectangle (\y,\x);
                \fi
                \ifnum\diag=4
                \fill[blue!, opacity=0.3] (\y-1,\x-1) rectangle (\y,\x);
                \fi
                \ifnum\diag=5
                \fill[blue!, opacity=0.3] (\y-1,\x-1) rectangle (\y,\x);
                \fi
                \ifnum\diag=6
                \fill[yellow!, opacity=0.3] (\y-1,\x-1) rectangle (\y,\x);
                \fi
                \ifnum\diag=7
                \fill[yellow!, opacity=0.3] (\y-1,\x-1) rectangle (\y,\x);
                \fi
                }
            }
        }
            \node[above] at (4.5,0) {N = 9 case};
        \end{scope}
    \end{tikzpicture}
    \caption{Example of this partition represented via a coloring on $N \times N$ grid.}
\end{figure}
We now establish that this partition satisfies the required conditions.
\begin{lemma}\label{PartitionLemma}
    Let $N \ge 4$ be given, then the sets $J_l$ defined previously for $1 \le l \le \floor{N/2}$ satisfy
    \[\rank \left[ \bfd_{i,j} \right]_{(i,j) \in J_l(N)} = 2N.\]
\end{lemma}
\begin{proof}
    We begin by recalling the equation used in Lemma \ref{RankGrowthLemma}, namely
    \[\rank (C^{\text{magic}}_{N})_{J} = \rank A_{J} - \dim(\im A_{J} \cap \im B) + E_N(J),\]
    where
    \[A_J = 
    \left[\begin{matrix}
        \bfe_{N}(i) \\
        \bfe_{N}(j)
    \end{matrix} \right]_{(i,j) \in J},\quad B = \left[\begin{matrix}
        {\bf1}_N & {\bf0}_N \\
        {\bf0}_N & {\bf1}_N
    \end{matrix}\right], \]
    and
    \[ E_N(J) = \dim\left( \im (C^{\text{magic}}_{N})_J \cap \im B \right).\]
    Note that if we can show that $\im A_{J_l(N)} = \im A$ and $E_N(J_l(N)) = 2$ for all $1 \le l \le \floor{N/2}$, then we are done because by our previous analysis this implies that 
    \[\rank A - \dim(\im A \cap \im B) = 2N-2,\]
    and hence $\rank (C^{\text{magic}}_{N})_{J} = 2N$.

    A few important observations about these sets $J_l(N)$ should be noted, for every $N \ge 4$ and $1 \le l \le \floor{N/2}$ the sets $J^{(1)}_{l}(N)$ and $J^{(2)}_{l}(N)$ both have size $N$ and contain a single element in every row and column. Additionally $J^{(1)}_{l}(N) \cap D_2(N) = \emptyset$ and 
    \[|J^{(2)}_{l}(N) \cap D_2(N)| = \begin{cases}
    2 & 2 \mid N \\
    1 & 2 \nmid N
    \end{cases}\]
    
    Note that due to the ``ladder like'' structure of $J_l(N)$ we may, via applications of the fundamental relation (\ref{FundRelation}), deduce that
    \[\im A_{J_l} = \im A.\]
    All that is left to show is that $E_N(J_l) = 2$ for all $1 \le l \le \floor{N/2}$. By the above observations it is not hard to see that
    \[\left[ \begin{matrix}
    {\bf1}_N \\
    {\bf0}_N
    \end{matrix} \right] \in \im \left( \rem(N,2)\sum_{(i,j) \in J^{(1)}_l(N)}\bfd_{i,j}+ \sum_{(i,j) \in J^{(2)}_l(N)}\bfd_{i,j} \right) \subset \im (C^{\text{magic}}_{N})_{J_l(N)},\]
    and
    \[\left[ \begin{matrix}
    {\bf0}_N \\
    {\bf1}_N
    \end{matrix} \right] \in \im\left( \sum_{(i,j) \in J^{(1)}_l(N)}\bfd_{i,j}, \left[ \begin{matrix}
    {\bf1}_N \\
    {\bf0}_N
    \end{matrix} \right] \right) \subset \im (C^{\text{magic}}_{N})_{J_l}(N),\]
    thus we establish that $E_N(J_l(N)) = 2$ and we are done.
\end{proof}
With this partition in place we may now fix a suitable bijection $\phi: [N^2] \to [N]^2$ which respects this partition. We begin by first fixing a bijection
\[\Tilde{\psi}: J_1(N) \to [2N].\]
We now use the important property that each of our partitions $J_l(N)$ are precisely a previous partition which has been shifted by two rows. Define 
\[\psi: \bigsqcup_{1 \le l \le K}J_l(N) \to [2NK]\] 
piecewise via the relation
\[\psi(i,j) = \Tilde{\psi}(\rem(i-2l+1)+1,j)+2Nl \text{ for }(i,j) \in J_l(N).\]

Finally, we take $\phi$ to be any bijection from $[N^2]$ to $[N]^2$ which satisfies
\begin{equation*}
    \phi \circ \psi^{-1} = \text{Id}_{[2NK]}.
\end{equation*}
One may refer to figure \ref{fig2} for an example of such a bijection. This bijection $\phi$ satisfies (\ref{MaxRank}), additionally we note that for any $1 \le n \le 2N$ the set
\[\{\phi(n),\phi(n+2N),\ldots,\phi(n+2N(K-1))\} \subset [N]^2,\]
lies in a single column. Thus, if we can find a \textbf{MMS}$(K,N)$, say $\bfZ \in \Z^{N \times N}$, which has distinct values along the columns we immediately have that (\ref{ResCond}) is satisfied and thus establish the existence of a nonsingular integer solution and hence trivially nonsingular local solutions.

\begin{figure}[!htbp]
    \centering
    \begin{tikzpicture}[scale=0.5, y=-1cm]
        \draw (0,0) grid (8,8);
        \foreach \x in {1,...,8}{
            \foreach \y in {1,...,8}{
            \pgfmathtruncatemacro\diag{Mod((\x-\y),8)}
            \ifnum\diag=0
            \fill[red!, opacity=0.3] (\y-1,\x-1) rectangle (\y,\x);
            \pgfmathtruncatemacro\entree{\y+\diag*8}
            \node at (\y-.5,\x-.5) {\entree};
            \fi
            \ifnum\diag=1
            \fill[red!, opacity=0.3] (\y-1,\x-1) rectangle (\y,\x);
            \pgfmathtruncatemacro\entree{\y+\diag*8}
            \node at (\y-.5,\x-.5) {\entree};
            \fi
            \ifnum\diag=2
            \fill[green!, opacity=0.3] (\y-1,\x-1) rectangle (\y,\x);
            \pgfmathtruncatemacro\entree{\y+\diag*8}
            \node at (\y-.5,\x-.5) {\entree};
            \fi
            \ifnum\diag=3
            \fill[green!, opacity=0.3] (\y-1,\x-1) rectangle (\y,\x);
            \pgfmathtruncatemacro\entree{\y+\diag*8}
            \node at (\y-.5,\x-.5) {\entree};
            \fi
            \ifnum\diag=4
            \fill[blue!, opacity=0.3] (\y-1,\x-1) rectangle (\y,\x);
            \pgfmathtruncatemacro\entree{\y+\diag*8}
            \node at (\y-.5,\x-.5) {\entree};
            \fi
            \ifnum\diag=5
            \fill[blue!, opacity=0.3] (\y-1,\x-1) rectangle (\y,\x);
            \pgfmathtruncatemacro\entree{\y+\diag*8}
            \node at (\y-.5,\x-.5) {\entree};
            \fi
            \ifnum\diag=6
            \fill[yellow!, opacity=0.3] (\y-1,\x-1) rectangle (\y,\x);
            \pgfmathtruncatemacro\entree{\y+\diag*8}
            \node at (\y-.5,\x-.5) {\entree};
            \fi
            \ifnum\diag=7
            \fill[yellow!, opacity=0.3] (\y-1,\x-1) rectangle (\y,\x);
            \pgfmathtruncatemacro\entree{\y+\diag*8}
            \node at (\y-.5,\x-.5) {\entree};
            \fi
            }
        }
        \node[above] at (4,0) {N = 8 case};
        
        \begin{scope}[xshift=10cm]
            \draw (0,0) grid (9,9);
        \foreach \x in {1,...,9}{
            \foreach \y in {1,...,9}{
            \pgfmathtruncatemacro\diag{Mod((\x-\y),9)}
            \pgfmathtruncatemacro\ycheck{Mod(\y,2)}
                \ifthenelse{\y<7}{
                    \ifthenelse{\ycheck=1}{
                    \ifnum\diag=0
                    \fill[red!, opacity=0.3] (\y,\x-1) rectangle (\y+1,\x);
                    \pgfmathtruncatemacro\entree{\y+\diag*9}
                    \node at (\y+.5,\x-.5) {\entree};
                    \fi
                    \ifnum\diag=1
                    \fill[red!, opacity=0.3] (\y,\x-1) rectangle (\y+1,\x);
                    \pgfmathtruncatemacro\entree{\y+\diag*9}
                    \node at (\y+.5,\x-.5) {\entree};
                    \fi
                    \ifnum\diag=2
                    \fill[green!, opacity=0.3] (\y,\x-1) rectangle (\y+1,\x);
                    \pgfmathtruncatemacro\entree{\y+\diag*9}
                    \node at (\y+.5,\x-.5) {\entree};
                    \fi
                    \ifnum\diag=3
                    \fill[green!, opacity=0.3] (\y,\x-1) rectangle (\y+1,\x);
                    \pgfmathtruncatemacro\entree{\y+\diag*9}
                    \node at (\y+.5,\x-.5) {\entree};
                    \fi
                    \ifnum\diag=4
                    \fill[blue!, opacity=0.3] (\y,\x-1) rectangle (\y+1,\x);
                    \pgfmathtruncatemacro\entree{\y+\diag*9}
                    \node at (\y+.5,\x-.5) {\entree};
                    \fi
                    \ifnum\diag=5
                    \fill[blue!, opacity=0.3] (\y,\x-1) rectangle (\y+1,\x);
                    \pgfmathtruncatemacro\entree{\y+\diag*9}
                    \node at (\y+.5,\x-.5) {\entree};
                    \fi
                    \ifnum\diag=6
                    \fill[yellow!, opacity=0.3] (\y,\x-1) rectangle (\y+1,\x);
                    \pgfmathtruncatemacro\entree{\y+\diag*9}
                    \node at (\y+.5,\x-.5) {\entree};
                    \fi
                    \ifnum\diag=7
                    \fill[yellow!, opacity=0.3] (\y,\x-1) rectangle (\y+1,\x);
                    \pgfmathtruncatemacro\entree{\y+\diag*9}
                    \node at (\y+.5,\x-.5) {\entree};
                    \fi
                    \ifnum\diag=8
                    \pgfmathtruncatemacro\entree{\y+\diag*9}
                    \node at (\y+.5,\x-.5) {\entree};
                    \fi
                    }{
                    \ifnum\diag=0
                    \fill[red!, opacity=0.3] (\y-2,\x-1) rectangle (\y-1,\x);
                    \pgfmathtruncatemacro\entree{\y+\diag*9}
                    \node at (\y-1.5,\x-.5) {\entree};
                    \fi
                    \ifnum\diag=1
                    \fill[red!, opacity=0.3] (\y-2,\x-1) rectangle (\y-1,\x);
                    \pgfmathtruncatemacro\entree{\y+\diag*9}
                    \node at (\y-1.5,\x-.5) {\entree};
                    \fi
                    \ifnum\diag=2
                    \fill[green!, opacity=0.3] (\y-2,\x-1) rectangle (\y-1,\x);
                    \pgfmathtruncatemacro\entree{\y+\diag*9}
                    \node at (\y-1.5,\x-.5) {\entree};
                    \fi
                    \ifnum\diag=3
                    \fill[green!, opacity=0.3] (\y-2,\x-1) rectangle (\y-1,\x);
                    \pgfmathtruncatemacro\entree{\y+\diag*9}
                    \node at (\y-1.5,\x-.5) {\entree};
                    \fi
                    \ifnum\diag=4
                    \fill[blue!, opacity=0.3] (\y-2,\x-1) rectangle (\y-1,\x);
                    \pgfmathtruncatemacro\entree{\y+\diag*9}
                    \node at (\y-1.5,\x-.5) {\entree};
                    \fi
                    \ifnum\diag=5
                    \fill[blue!, opacity=0.3] (\y-2,\x-1) rectangle (\y-1,\x);
                    \pgfmathtruncatemacro\entree{\y+\diag*9}
                    \node at (\y-1.5,\x-.5) {\entree};
                    \fi
                    \ifnum\diag=6
                    \fill[yellow!, opacity=0.3] (\y-2,\x-1) rectangle (\y-1,\x);
                    \pgfmathtruncatemacro\entree{\y+\diag*9}
                    \node at (\y-1.5,\x-.5) {\entree};
                    \fi
                    \ifnum\diag=7
                    \fill[yellow!, opacity=0.3] (\y-2,\x-1) rectangle (\y-1,\x);
                    \pgfmathtruncatemacro\entree{\y+\diag*9}
                    \node at (\y-1.5,\x-.5) {\entree};
                    \fi
                    \ifnum\diag=8
                    \pgfmathtruncatemacro\entree{\y+\diag*9}
                    \node at (\y-1.5,\x-.5) {\entree};
                    \fi
                    }
                }{
                \ifnum\diag=0
                \fill[red!, opacity=0.3] (\y-1,\x-1) rectangle (\y,\x);
                \pgfmathtruncatemacro\entree{\y+\diag*9}
                \node at (\y-.5,\x-.5) {\entree};
                \fi
                \ifnum\diag=1
                \fill[red!, opacity=0.3] (\y-1,\x-1) rectangle (\y,\x);
                \pgfmathtruncatemacro\entree{\y+\diag*9}
                \node at (\y-.5,\x-.5) {\entree};
                \fi
                \ifnum\diag=2
                \fill[green!, opacity=0.3] (\y-1,\x-1) rectangle (\y,\x);
                \pgfmathtruncatemacro\entree{\y+\diag*9}
                \node at (\y-.5,\x-.5) {\entree};
                \fi
                \ifnum\diag=3
                \fill[green!, opacity=0.3] (\y-1,\x-1) rectangle (\y,\x);
                \pgfmathtruncatemacro\entree{\y+\diag*9}
                \node at (\y-.5,\x-.5) {\entree};
                \fi
                \ifnum\diag=4
                \fill[blue!, opacity=0.3] (\y-1,\x-1) rectangle (\y,\x);
                \pgfmathtruncatemacro\entree{\y+\diag*9}
                \node at (\y-.5,\x-.5) {\entree};
                \fi
                \ifnum\diag=5
                \fill[blue!, opacity=0.3] (\y-1,\x-1) rectangle (\y,\x);
                \pgfmathtruncatemacro\entree{\y+\diag*9}
                \node at (\y-.5,\x-.5) {\entree};
                \fi
                \ifnum\diag=6
                \fill[yellow!, opacity=0.3] (\y-1,\x-1) rectangle (\y,\x);
                \pgfmathtruncatemacro\entree{\y+\diag*9}
                \node at (\y-.5,\x-.5) {\entree};
                \fi
                \ifnum\diag=7
                \fill[yellow!, opacity=0.3] (\y-1,\x-1) rectangle (\y,\x);
                \pgfmathtruncatemacro\entree{\y+\diag*9}
                \node at (\y-.5,\x-.5) {\entree};
                \fi
                \ifnum\diag=8
                \pgfmathtruncatemacro\entree{\y+\diag*9}
                \node at (\y-.5,\x-.5) {\entree};
                \fi
                }
            }
        }
            \node[above] at (4.5,0) {N = 9 case};
        \end{scope}
    \end{tikzpicture}
    \caption{Example of a bijection $\phi$ satisfying the above properties represented via labeling the $(i,j)$th entree on $N \times N$ grid the value $\phi^{-1}(i,j)$.}\label{fig2}
\end{figure}

Recall from section \ref{sec:intro} that \textbf{DDLS}$(N)$ exist for $N \ge 4$ and are trivially \textbf{MMS}$(K,N)$ which satisfy this column condition. Hence we are done and have established Theorem \ref{MainThm}.

\pagebreak
\bibliographystyle{amsbracket}
\providecommand{\bysame}{\leavevmode\hbox to3em{\hrulefill}\thinspace}

\end{document}